\documentclass[a4paper, 11pt]{amsart}
\usepackage{amsmath,amsthm,amssymb,color}
\usepackage[arrow,matrix]{xy}
\usepackage{hyperref}

\theoremstyle{ams}
\newtheorem{theorem}{Theorem}[section]
\newtheorem{proposition}[theorem]{Proposition}
\newtheorem{lemma}[theorem]{Lemma}
\newtheorem{corollary}[theorem]{Corollary}
\numberwithin{equation}{section}

\theoremstyle{definition}

\newtheorem{remark}[theorem]{Remark}
\newtheorem{example}[theorem]{Example}
\newtheorem{question}[theorem]{Question}

\newcommand{\C}{\mathbb{C}}
\newcommand{\CC}{\underline{\mathbb{C}}}
\newcommand{\Q}{\mathbb{Q}}

\newcommand{\Z}{\mathbb{Z}}
\newcommand{\CP}{\mathbb{C}P}
\newcommand{\va}{\mathbf{a}}
\newcommand{\vb}{\mathbf{b}}

\begin{document}
\title[$\Q$-trivial generalized Bott manifolds]{$\Q$-trivial generalized Bott manifolds}

\author[S.Park]{Seonjeong Park}
\address{Department of Mathematical Sciences, KAIST, 291 Daehak-ro, Yuseong-gu, Daejeon 305-701, Korea}
\email{psjeong@kaist.ac.kr}

\author[D.Y.Suh]{Dong Youp Suh}
\address{Department of Mathematical Sciences, KAIST, 291 Daehak-ro, Yuseong-gu, Daejeon 305-701, Korea}
\email{dysuh@math.kaist.ac.kr}

\thanks{The first author is supported by the second stage of the Brain Korea 21 Project, the Development of Project of Human Resources in Mathematics, KAIST in 2012.}
\thanks{The second author is partially supported by Basic Science Research Program through the national Research Foundation of Korea(NRF) founded by the Ministry of Education, Science and Technology (2012-0000795).}


\date{\today}
\maketitle
\begin{abstract}
    When the cohomology ring of a generalized Bott manifold with $\Q$-coefficient is isomorphic to that of a product of complex projective spaces $\CP^{n_i}$, the generalized Bott manifold is said to be $\Q$-trivial. We find a necessary and sufficient condition for a generalized Bott manifold to be $\Q$-trivial. In particular, every $\Q$-trivial generalized Bott manifold is diffeomorphic to a $\prod_{n_i>1}\CP^{n_i}$-bundle over a $\Q$-trivial Bott manifold.
\end{abstract}

\section{Introduction}

    A \emph{generalized Bott tower of height $h$} is a sequence of complex projective space bundles
    \begin{equation}\label{def:GB}
        B_h\stackrel{\pi_h}\longrightarrow B_{h-1} \stackrel{\pi_{h-1}}\longrightarrow
        \dots \stackrel{\pi_2}\longrightarrow B_1 \stackrel{\pi_1}\longrightarrow
        B_0=\{\text{a point}\},
    \end{equation}
    where $B_i=P(\underline{\C}\oplus\xi_i)$, $\underline{\C}$ is a trivial complex line bundle, $\xi_i$ is a Whitney sum of $n_i$ complex line bundles over $B_{i-1}$, and $P(\cdot)$ stands a projectivization. Each $B_i$ is called an \emph{$i$-stage generalized Bott manifold}. When all $n_i$'s are $1$ for $i=1,\ldots,h$, the sequence~\eqref{def:GB} is called a \emph{Bott tower of height $h$} and $B_i$ is called an \emph{$i$-stage Bott manifold}.

    A ($h$-stage) generalized Bott manifold is said to be \emph{$\Q$-trivial} (respectively, \emph{$\Z$-trivial}) if $H^\ast(B_h;\Q) \cong H^\ast(\prod_{i=1}^h \CP^{n_i};\Q)$ (respectively, $H^\ast(B_h;\Z) \cong H^\ast(\prod_{i=1}^h \CP^{n_i};\Z)$). It is shown in~\cite{CMS10b} that if $B_h$ is $\Z$-trivial, then every fiber bundle in the tower~\eqref{def:GB} is trivial so that $B_h$ is diffeomorphic to $\prod_{i=1}^h \CP^{n_i}$. Furthermore, Choi and Masuda show that every ring isomorphism between $\Z$-cohomology rings of two $\Q$-trivial Bott manifolds is induced by some diffeomorphism between them (see Theorem \ref{thm:CM} and \cite{CM}).

    We find a necessary and sufficient condition for a generalized Bott manifold to be $\Q$-trivial. Namely, we have the following proposition.

    \begin{proposition}\label{prop:Q-trivial}
        An $h$-stage generalized Bott manifold $B_h$ is $\Q$-trivial if and only if each vector bundle $\xi_i$, $i=1,\ldots,h$, satisfies
        \begin{equation}\label{relations: chern classes}
            (n_i+1)^kc_k(\xi_{i})={n_i+1\choose k}c_1(\xi_{i})^k
        \end{equation}
        for $k=1,\ldots,n_i+1$, where $B_i=P(\CC\oplus\xi_i)$.
    \end{proposition}

    Moreover, the following theorem says that a $\Q$-trivial generalized Bott manifold without $\CP^1$-fibration is weakly equivariantly diffeomorphic to a trivial generalized Bott manifold.

    \begin{theorem}\label{thm:equivalent conditions}
        Let $B_h$ be a generalized Bott manifold such that all $n_i$'s are greater than $1$. Then the following are equivalent
        \begin{enumerate}
            \item $B_h$ is $\Q$-trivial,\label{thm:equiv_condition 1}
            \item total Chern class $c(\xi_i)$ is trivial for each $i=1,\ldots,h$,\label{thm:equiv_condition 2}
            \item $B_h$ is $\Z$-trivial, and\label{thm:equiv_condition 3}
            \item $B_h$ is weakly equivariantly diffeomorphic to the product of projective spaces $\prod_{i=1}^h \CP^{n_i}$.\label{thm:equiv_condition 4}
        \end{enumerate}
    \end{theorem}

    In the light of Theorem~\ref{thm:equivalent conditions}, we have a natural question.

    \begin{question}\label{question:not Z-isomorphic}
        Let $B_h$ and $B'_h$ be generalized Bott manifolds with $n_i>1$, $i=1,\ldots,h$. Is $H^\ast(B_h;\Z)$ isomorphic to $H^\ast(B'_h;\Z)$ if $H^\ast(B_h;\Q) \cong H^\ast(B'_h;\Q)$?
    \end{question}
    Unfortunately, Example~\ref{example:not Z-isomorphic} shows that the answer to the question is negative.

    From the proposition, we can deduce the following theorem.
    \begin{theorem}\label{thm:Q-trivial gen. Bott}
        Every $\Q$-trivial generalized Bott manifold is diffeomorphic to a $\prod_{n_i>1}\CP^{n_i}$-bundle over a $\Q$-trivial Bott manifold.
    \end{theorem}
%
%
%


    The remainder of this paper is organized as follows. In section~\ref{sec:cohomology ring of a generalized Bott manifold}, we recall general facts on a generalized Bott manifold and deal with its cohomology ring. In section~\ref{sec:Q-trivial generalized Bott manifolds}, we prove Proposition~\ref{prop:Q-trivial}, Theorem~\ref{thm:equivalent conditions}, and Theorem~\ref{thm:Q-trivial gen. Bott}.

\section{Cohomology ring of a generalized Bott manifold}\label{sec:cohomology ring of a generalized Bott manifold}

    Let $B$ be a smooth manifold and let $E$ be a complex vector bundle over $B$. Let $P(E)$ denote the projectivization of $E$. Let $y \in H^2(P(E))$ be the negative of the first Chern class of the tautological line bundle over $P(E)$. Then $H^\ast(P(E))$ can be viewed as an algebra over $H^\ast(B)$ via $\pi^\ast\colon H^\ast(B) \rightarrow H^\ast(P(E))$, where $\pi\colon P(E) \rightarrow B$ denotes the projection. When $H^\ast(B)$ is finitely generated and torsion free (this is the case when $B$ is a toric manifold), $\pi^\ast$ is injective and $H^\ast(P(E))$  as an algebra over $H^\ast(B)$ is known to be described as
    \begin{equation}\label{B-H formula}
        H^\ast(P(E))=H^\ast(B)[y]\left/\left\langle\sum_{k=0}^n c_k(E)y^{n-k}\right\rangle\right.,
    \end{equation}
    where $n$ denotes the complex dimension of the fiber of $E$ (see~\cite{BH}).

    For a generalized Bott manifold $B_h$ in~\eqref{def:GB}, since $\pi_j^\ast\colon H^\ast(B_{j-1}) \to H^\ast(B_j)$ is injective, we regard $H^\ast(B_{j-1})$ as a subring of $H^\ast(B_j)$ for each $j$ so that we have a filtration
    $$H^\ast(B_h) \supset H^\ast(B_{h-1}) \supset \cdots \supset H^\ast(B_1).$$ Let $x_j \in H^2(B_j)$ denote minus the first Chern class of the tautological line bundle over $B_j=P(\underline{\C}\oplus\xi_j)$. We may think of $x_j$ as an element of $H^2(B_i)$ for $i \geq j$. Then the repeated use of~\eqref{B-H formula} shows that the ring structure of $H^\ast(B_h)$ can be described as
    \begin{equation*}
        H^\ast(B_h)=\Z[x_1,\ldots,x_h]\left/\left\langle x_i^{n_i+1}+c_1(\xi_i)x_i^{n_i}+\cdots+c_{n_i}(\xi_i)x_i \middle| i=1,\ldots,h \right\rangle.\right.
    \end{equation*}
    Let $\xi_{2,1}$ be the tautological line bundle over $B_1=\CP^{n_1}$ and let $\xi_{3,1}=\pi_2^\ast(\xi_{2,1})$ the pull-back bundle of the tautological line bundle over $B_1$ to $B_2$ via the projection $\pi_2\colon B_2 \to B_1$. In general, let $\xi_{j,j-1}$ be the tautological line bundle over $B_{j-1}$ and we define inductively
    $$\xi_{j,j-k}=\pi_{j-1}^\ast\circ\cdots\circ\pi_{j-k+1}^\ast(\xi_{j-k+1,j-k})$$
    for $k=2,\ldots,j-1$. Then one can see that the Whitney sum of complex line bundles $\xi_i$ over $B_{i-1}$ in the sequence~\eqref{def:GB} can be written as $$\xi_i:=(\xi_{i,1}^{a_{11}^i}\otimes\cdots\otimes\xi_{i,i-1}^{a_{1,i-1}^i})\oplus
    \cdots\oplus(\xi_{i,1}^{a_{n_i,1}^{i}}\otimes\cdots\otimes\xi_{i,i-1}^{a_{n_i,i-1}^i})$$ for some integers $a_{11}^i,\ldots,a_{n_i,i-1}^i$. Note that $\xi_1=(\underline{\C})^{n_1}$. Hence, the total Chern class of $\xi_i$ is
    \begin{equation}\label{eq:chern class formula}
        c(\xi_i)=\prod_{j=1}^{n_i}\left(1+\sum_{k=1}^{i-1}a_{jk}^ix_k\right).
    \end{equation}
    Therefore, the cohomology ring of $B_h$ is
    \begin{equation}\label{cohomology ring of B_h}
        \begin{split}
        &H^\ast(B_h;\Z)\\
        &=\Z[x_1,\ldots,x_h]\left/ \left\langle x_i^{n_i+1}+c_1(\xi_i)x_i^{n_i}+\cdots+c_{n_i}(\xi_i)x_i\middle| i=1,\ldots,h \right\rangle\right.\\
        &=\Z[x_1,\ldots,x_h]\left/\left\langle x_i\prod_{j=1}^{n_i}(\sum_{k=1}^{i-1}a_{jk}^ix_k+x_i)\middle| i=1,\ldots,h\right\rangle.\right.
        \end{split}
    \end{equation}

    \begin{remark}\label{rmk:associated matrix}
        We can associate a generalized Bott manifold $B_h$ with an $h \times h$ vector matrix $A$ as follows:
            \begin{equation}\label{associated matrix}
            A^T = \left(\begin{array}{cccc}
            \mathbf{1} & & & \\
            \mathbf{a}_1^2&\mathbf{1}&&\\
            \vdots&\vdots&\ddots& \\
            \mathbf{a}_1^h&\mathbf{a}_2^h&\cdots&\mathbf{1}
            \end{array}\right),
        \end{equation}
        where $$\mathbf{a}_k^i=\left(\begin{array}{c}a_{1k}^i\\\vdots\\a_{n_ik}^i\end{array}\right) \mbox{ and } \mathbf{1}=\left(\begin{array}{c}1\\\vdots\\1\end{array}\right).$$ Moreover we can consider $B_h$ as a quasitoric manifold over the product of simplices $\prod_{i=1}^h \Delta^{n_i}$ with the reduced characteristic matrix $\Lambda_\ast=-A^T$.
    \end{remark}

\section{$\Q$-trivial generalized Bott manifolds}\label{sec:Q-trivial generalized Bott manifolds}

    As we mentioned in the introduction, Choi and Masuda classify $\Q$-trivial Bott manifolds as follows.
    \begin{theorem}\cite{CM}\label{thm:CM}
        \begin{enumerate}
            \item A Bott manifold $B_h$ is $\Q$-trivial if and only if for each $i=1,\ldots,h$, each line bundle $\xi_i$ satisfies $c_1(\xi_i)^2=0$ in $H^\ast(B_h;\mathbb{Z})$.
            \item Every ring isomorphism $\varphi$ between two $\Q$-trivial Bott manifolds $B_h$ and $B_h'$ is induced by some diffeomorphism $B_h \to B_h'$.
        \end{enumerate}
    \end{theorem}

    In this section we shall prove Proposition~\ref{prop:Q-trivial} and Theorem~\ref{thm:equivalent conditions}. To prove them, we need the following lemmas.

    \begin{lemma}\label{lem:Q-triviality}
        If a generalized Bott manifold $B_h$ is $\Q$-trivial, then there exist linearly independent primitive elements $z_1,\ldots,z_h$ in $H^2(B_h;\Z)$ such that $z_i^{n_i}$ is not zero but $z_i^{n_i+1}$ is zero in $H^\ast(B_h;\Z)$ for $i=1,\ldots,h$.
    \end{lemma}
    \begin{proof}
        Let $H^\ast(B_h;\Z)$ be generated by $x_1,\ldots,x_h$ as in~\eqref{cohomology ring of B_h} and let
        $$H^\ast(\prod_{i=1}^h B_h;\Q)=\Q[y_1,\ldots,y_h]\left/\left\langle y_i^{n_i+1}\middle| i=1,\ldots,h \right\rangle.\right.$$ Since both $\{x_1,\ldots,x_h\}$ and $\{y_1,\ldots,y_h\}$ are sets of generators of $H^2(B_h;\Q)$, we can write $$y_i=\sum_{j=1}^h c_{ij}x_j\quad \mbox{ for $i=1,\ldots,h$ and $c_{ij}\in\Q$},$$ where the determinant of the matrix $C=(c_{ij})_{h\times h}$ is non-zero. We may assume that $c_{ij}$'s are irreducible fractions. Multiplying $(c_{i,1},\ldots,c_{i,h})$ by the least common denominator $r_i$ of a set $\{c_{i,1},\ldots,c_{i,h}\}$, we can get a primitive element $z_i=r_iy_i=r_i\sum_{j=1}^h c_{ij}x_j$ in $H^2(B_h;\Z)$ such that $z_i^{n_i+1}=r_i^{n_i+1}y_i^{n_i+1}$ is zero in $H^\ast(B_h;\Z)$ for each $i=1,\ldots,h$. Since the elements $y_1,\ldots,y_h$ are linearly independent, the elements $z_1,\ldots,z_h$ are also linearly independent. Since $y_i^{n_i}$ is not zero in $H^\ast(B_h;\Q)$, $z_i^{n_i}$ cannot be zero in $H^\ast(B_h;\Z)$. This proves the lemma.
    \end{proof}

    \begin{lemma}\cite{CMS10b}\label{lem:one-dimensional}
        Let $B_m$ be an $m$-stage generalized Bott manifold. Then
        the set $$\left\{bx_m+w\in H^2(B_m) \middle| 0 \neq b \in \Z,\,w\in H^2(B_{m-1}),\,(bx_m+w)^{n_m+1}=0\right\}$$ lies in a one-dimensional subspace of $H^2(B_m)$ if it is non-empty.
    \end{lemma}
    \begin{proof}
        To satisfy $(bx_m+w)^{n_m+1}=0$, we need $bc_1(\xi_m)=(n_m+1)w$.
    \end{proof}

    \begin{lemma}\cite{CMS10b}\label{lem:non-zero coefficient}
        For an element $z=\sum_{i=1}^h b_ix_i \in H^2(B_h)$, if $b_i$ is non-zero, then $z^{n_i}$ cannot be zero in $H^\ast(B_h)$.
    \end{lemma}
    \begin{proof}
        If we expand $(\sum_{i=1}^h b_ix_i)^{n_i}$, there appears a non-zero scalar multiple of $x_i^{n_i}$ because $b_i\neq 0$. Then, $z^{n_i}$ cannot belong to the ideal generated by the polynomials $x_i\prod_{j=1}^{n_i}(\sum_{k=1}^{i-1}a_{jk}^ix_k+x_i)$, hence it is not zero in $H^\ast(B_h)$.
    \end{proof}

    Now we can prove Proposition~\ref{prop:Q-trivial}.
    \begin{proof}[Proof of Proposition~\ref{prop:Q-trivial}]
        If each vector bundle $\xi_i$ satisfies the conditions~\eqref{relations: chern classes}, then $\left(x_i+\frac{1}{n_i+1}c_1(\xi_i)\right)^{n_i+1}$ is zero in $H^\ast(B_h;\Q)$. Since the set
        $$\left\{x_i+\frac{1}{n_i+1}c_1(\xi_i)\middle| i=1,\ldots,h\right\}$$
        generates $H^\ast(B_h;\Q)$ as a graded ring, this shows that $B_h$ is $\Q$-trivial.

        Conversely, if a generalized Bott manifold is $\Q$-trivial, then there are linearly independent and primitive elements $z_1,\ldots,z_h$ in $H^2(B_h;\Z)$ such that $z_i^{n_i+1}$ is zero but $z_i^{n_i}$ is not zero in $H^\ast(B_h)$ by Lemma~\ref{lem:Q-triviality}. We can put $z_i=\sum_{j=1}^h b_{ij}x_j$ with $b_{ij} \in \Z$ for each $i=1,\ldots,h$.

        Now, consider a map $\mu\colon \{1,\ldots,h\} \to \mathbb{N}$ given by $j \mapsto n_j$. Further assume that the image of $\mu$ is the set $\{N_1,\ldots,N_m\}$ with $N_1 < \cdots < N_m$. We will show inductively that each $z_i$ can be written as $r_i\left(x_i+\frac{1}{\mu(i)+1}c_1(\xi_i)\right)$ for some $r_i\in\Z\setminus\{0\}$.

        $\underline{Case~1}:$ Assume $i \in \mu^{-1}(N_1)$. Let $\mu^{-1}(N_1):=\{i_1,\ldots,i_\alpha\}$ with $i_1<\cdots<i_\alpha$. We have $z_i^{N_1+1}=0$. Then, by Lemma~\ref{lem:non-zero coefficient}, we can see that
        \begin{equation}\label{eqn1:when mu(i)=N_1}
            z_i=\sum_{j \in \mu^{-1}(N_1)}b_{ij}x_j,
        \end{equation} that is, $b_{ij'}=0$ for $j'\not\in \mu^{-1}(N_1)$. Note that for each $i\in\mu^{-1}(N_1)$, one of $b_{ij}$'s is nonzero for $j\in\mu^{-1}(N_1)$ because the set $\{z_i\,|\,i\in\mu^{-1}(N_1)\}$ is linearly independent. For some $i_p\in\mu^{-1}(N_1)$, if $b_{i_p i_\alpha}$ is nonzero, then $z_{i_p}\in H^2(B_{i_\alpha})$ and $b_{ii_\alpha}=0$  for all $i\in\mu^{-1}(N_1)\setminus\{i_p\}$ by Lemma~\ref{lem:one-dimensional}. Put $w_{i_\alpha}:=z_{i_p}$. If $b_{i_q i_{\alpha-1}}$ is nonzero for some $i_q\in \mu^{-1}(N_1)\setminus\{i_p\}$, then $z_{i_q}\in H^2(B_{i_{\alpha-1}})$ and $b_{ii_{\alpha-1}}=0$ for all $i\in\mu^{-1}(N_1)\setminus\{i_p,i_q\}$. Now, put $w_{i_{\alpha-1}}:=z_{i_q}$. In this way, for each $i\in\mu^{-1}(N_1)$, we can obtain $w_i\in H^2(B_i)$ such that $w_i\not\in H^2(B_{i-1})$ and $w_i^{N_1+1}=0$ in $H^\ast(B_h)$. Moreover, from the proof of Lemma~\ref{lem:one-dimensional}, we can write
        \begin{equation}\label{eqn2:when mu(i)=N_1}
            w_i :=r_i\left(x_i+\frac{1}{N_1+1}c_1(\xi_i)\right) \in H^2(B_i)
        \end{equation} for each $i \in \mu^{-1}(N_1)$.
        In particular, if $N_1=1$, then $w_i$ is of the form either $\pm x_i$ or $\pm(2x_i+c_1(\xi_i))$ for $i\in\mu^{-1}(N_1)$. Furthermore, without loss of generality, we may assume that $z_i=w_i$ for $i\in\mu^{-1}(N_1)$.

        $\underline{Case~2}:$ Assume that $z_k=r_k\left(x_k+\frac{1}{\mu(k)+1}c_1(\xi_k)\right)$ for $N_1\leq \mu(k)\leq N_{n-1}$ and let $\ell\in\mu^{-1}(N_{n})$. Then we have $z_\ell^{N_n+1}=0$ . Then by Lemma~\ref{lem:non-zero coefficient}, we can easily see that
        $$z_\ell=\sum_{k \in \mu^{-1}(N_{<n})} b_{\ell k}x_k+\sum_{j \in \mu^{-1}(N_n)} b_{\ell j}x_j,$$ where $N_{<n}=\{N_1,\ldots,N_{n-1}\}$. That is, $b_{\ell j'}=0$ for $j'\not\in\mu^{-1}(N_{\leq n})$. Since $z_\ell^{N_n+1}$ is zero in $H^\ast(B_h)$, we have
        \begin{equation}\label{eqn1:when m(i)=N_n}
            \begin{split}
                &\left(\sum_{k \in \mu^{-1}(N_{<n})} b_{\ell k}x_k+\sum_{j \in \mu^{-1}(N_n)} b_{\ell j}x_j\right)^{N_n+1}\\
                &\quad=\sum_{k\in\mu^{-1}(N_{<n})} f_k(x_1,\ldots,x_h)(x_k^{\mu(k)+1}+c_1(\xi_k)x_k^{\mu(k)}+\cdots+c_{\mu(k)}(\xi_k)x_k)\\
                &\qquad+\sum_{j\in\mu^{-1}(N_n)}b_{\ell j}^{N_n+1} (x_j^{N_n+1}+c_1(\xi_j)x_j^{N_n}+\cdots+c_{N_n}(\xi_j)x_j)\\
            \end{split}
        \end{equation} as polynomials, where $f_k(x_1,\ldots,x_h)$ is a homogeneous polynomial of degree $N_n-\mu(k)$ for each $k\in\mu^{-1}(N_{<n})$. Note that for each $\ell\in\mu^{-1}(N_n)$, one of $b_{\ell j}$'s is non-zero for $j\in\mu^{-1}(N_n)$ from the linearly independency of the set $\{z_i\,|\,i\in \mu^{-1}(N_{\leq n})\}$. Let $\mu^{-1}(N_n):=\{\ell_1,\ldots,\ell_\beta\}$ with $\ell_1<\cdots<\ell_\beta$. Assume $b_{\ell_p\ell_\beta}$ is nonzero for some $\ell_p\in\mu^{-1}(N_n)$. Substituting $\ell=\ell_p$ into \eqref{eqn1:when m(i)=N_n} and comparing the monomials containing $x_{\ell_\beta}^{N_n}$ as a factor on both sides of~\eqref{eqn1:when m(i)=N_n}, we have
        \begin{equation*}
            (N_n+1)\left(\sum_{k\in\mu^{-1}(N_{<n})}b_{\ell_p k}x_k+\sum_{j\in\mu^{-1}(N_n)\atop j \neq \ell_\beta}b_{\ell_p j}x_j\right)=b_{\ell_p \ell_\beta}^{N_n+1}c_1(\xi_{\ell_\beta}).
        \end{equation*}
        Since $c_1(\xi_{\ell_\beta})$ belongs to $H^2(B_{\ell_\beta -1})$, we can see that $b_{\ell_p k}=0$ for $k>\ell_\beta$. That is,
        $$z_{\ell_p}=\sum_{k\in\mu^{-1}(N_{<n})\atop k<\ell_\beta}b_{\ell_p k}x_k+\sum_{j\in\mu^{-1}(N_n)}b_{\ell_p j}x_j.$$ Thus, we can see that $z_{\ell_p}\in H^2(B_{\ell_\beta})$ and $b_{\ell\ell_\beta}=0$ for all $\ell\in \mu^{-1}(N_n)\setminus\{\ell_p\}$ by Lemma~\ref{lem:one-dimensional}. Put $w_{\ell_\beta}:=z_{\ell_p}$. Now assume that $b_{\ell_q\ell_{\beta-1}}$ is nonzero for some $\ell_q\in\mu^{-1}(N_n)\setminus\{\ell_p\}$. Substituting $\ell=\ell_q$ into \eqref{eqn1:when m(i)=N_n} and comparing the monomials containing $x_{\ell_{\beta-1}}^{N_n}$ as a factor on both sides of \eqref{eqn1:when m(i)=N_n}, we have
        \begin{equation*}
            (N_n+1)\left(\sum_{k\in\mu^{-1}(N_{<n})}b_{\ell_q k}x_k+\sum_{j\in\mu^{-1}(N_n)\atop j < \ell_{\beta-1}}b_{\ell_q j}x_j\right)=b_{\ell_q \ell_{\beta-1}}^{N_n+1}c_1(\xi_{\ell_{\beta-1}}).
        \end{equation*}
        Since $c_1(\xi_{\ell_{\beta-1}})$ belongs to $H^2(B_{\ell_{\beta-1}-1})$, we can see that $b_{\ell_qk}=0$ for $k>\ell_{\beta-1}$, and hence,
        $$z_{\ell_q}=\sum_{k\in\mu^{-1}(N_{<n})\atop k<\ell_{\beta-1}}b_{\ell_p k}x_k+\sum_{j\in\mu^{-1}(N_n)\atop j<\ell_\beta}b_{\ell_p j}x_j.$$
        Thus, we can see that $z_{\ell_q}\in H^2(B_{\ell_{\beta-1}})$ and $b_{\ell\ell_{\beta-1}}=0$ for all $\ell\in\mu^{-1}(N_n)\setminus\{\ell_p,\ell_q\}$ by Lemma~\ref{lem:one-dimensional}. Now, put $w_{\ell_{\beta-1}}:=z_{\ell_q}$. In this way, for each $\ell\in\mu^{-1}(N_n)$, we can obtain $w_\ell\in H^2(B_\ell)$ such that $w_\ell\not\in H^2(B_{\ell-1})$ and $w_\ell^{N_n+1}=0$ in $H^2(B_h)$. Moreover, from the proof of Lemma~\ref{lem:one-dimensional}, $w_\ell$ can be written as $r_\ell\left(x_\ell+\frac{1}{N_n+1}c_1(\xi_\ell)\right)$. Furthermore, without loss of generality, we may assume that $z_\ell=w_\ell$ for $\ell\in \mu^{-1}(N_n)$.

        By Cases 1 and 2, we can see that, for each $i=1,\ldots,h$, we can write $$z_i=r_i\left(x_i+\frac{1}{n_i+1}c_1(\xi_i)\right)$$ for some $r_i\in\Z\setminus\{0\}$.
        Therefore, $\{(n_i+1)x_i+c_1(\xi_i)\}^{n_i+1}$ is zero in $H^\ast(B_h)$.
        From this, we can see
        \begin{equation*}
            (n_i+1)^kc_k(\xi_{i})={n_i+1\choose k}c_1(\xi_{i})^k\mbox{ and } c_1(\xi_{i})^{n_i+1}=0
        \end{equation*}
        $k=1,\ldots,n_i$.
    \end{proof}
    Hence, the above theorem implies the statement (1) of Theorem \ref{thm:CM}.

    By using Proposition~\ref{prop:Q-trivial}, we can prove Theorem~\ref{thm:equivalent conditions}.

    \begin{proof}[Proof of Theorem~\ref{thm:equivalent conditions}]
        We first prove the implication \eqref{thm:equiv_condition 1}$\Rightarrow$\eqref{thm:equiv_condition 2}. By Proposition~\ref{prop:Q-trivial}, we have the relation
        \begin{equation}\label{eq:relation between c_1 and c_2}
            (n_i+1)^2c_2(\xi_i)=\frac{n_i(n_i+1)}{2}c_1(\xi_i)^2.
        \end{equation}
        If $n_i=2$, from \eqref{eq:chern class formula} and \eqref{eq:relation between c_1 and c_2}, we have
        \begin{equation}\label{eqn:c_1,c_2 when n=2}
            \begin{split}
            &\{(a_{11}^ix_1+\cdots+a_{1,i-1}^ix_{i-1})+(a_{21}^ix_1+\cdots+a_{2,i-1}^ix_{i-1})\}^2\\
            &=3(a_{11}^ix_1+\cdots+a_{1,i-1}^ix_{i-1})(a_{21}^ix_1+\cdots+a_{2,i-1}^ix_{i-1}).
            \end{split}
        \end{equation}
        For $j=1,\ldots,i-1$, since $x_j^2 \neq 0$ in $H^\ast(B_i)$, by comparing the coefficients of $x_j^2$ on both sides of \eqref{eqn:c_1,c_2 when n=2}, we have $(a_{1j}^i+a_{2j}^i)^2=3a_{1j}^ia_{2j}^i$ whose integer solution is only $a_{1j}^i=a_{2j}^i=0$. If $n_i=n>2$, then we have
        \begin{equation}\label{eqn:c_1,c_2 when n>2}
            \begin{split}
            &n\{(a_{11}^ix_1+\cdots+a_{1,i-1}^ix_{i-1})+\cdots+(a_{21}^ix_1+\cdots+a_{2,i-1}^ix_{i-1})\}^2\\
            &=2(n+1)\{(a_{11}^ix_1+\cdots+a_{1,i-1}^ix_{i-1})(a_{21}^ix_1+\cdots+a_{2,i-1}^ix_{i-1})\\
            &\quad+\cdots+(a_{n-1,1}^ix_1+\cdots+a_{n-1,i-1}^ix_{i-1})(a_{n,1}^ix_1+\cdots+a_{n,i-1}^ix_{i-1})\}.
            \end{split}
        \end{equation}
        Since $x_j^2 \neq 0$ in $H^\ast(B_i)$ for $j=1,\ldots,i-1$, by comparing the coefficients of $x_j^2$ on both sides of \eqref{eqn:c_1,c_2 when n>2} we have
        \begin{equation}\label{eqn:c_1,c_2 when n>2(2)}
            n(a_{i,j}^i+\cdots+a_{nj}^i)^2=2(n+1)\sum_{1\leq k < \ell \leq n}a_{kj}^ia_{\ell j}^i.
        \end{equation}
        The equation~\eqref{eqn:c_1,c_2 when n>2(2)} is equivalent to
        \begin{equation*}
            (n-2)\{(a_{1j}^i)^2+\cdots+(a_{nj}^i)^2\}+(a_{1j}^i-a_{2j}^i)^2+\cdots+(a_{n-1,j}^i-a_{nj}^i)^2=0.
        \end{equation*}
        Since $n>2$, we can see that $a_{1j}^i=\cdots=a_{nj}^i=0$ for each $j=1,\ldots,i-1$. Therefore, in any case, $c(\xi_i)$ is trivial for all $i=1,\ldots,h$.

        The implications \eqref{thm:equiv_condition 2}$\Rightarrow$\eqref{thm:equiv_condition 3} and \eqref{thm:equiv_condition 3}$\Rightarrow$\eqref{thm:equiv_condition 1} are clear.

        The implication \eqref{thm:equiv_condition 3}$\Leftrightarrow$\eqref{thm:equiv_condition 4} is proved by Choi-Masuda-Suh~\cite{CMS10b}.

        Therefore, all four conditions are equivalent.
    \end{proof}
    From the proof of \eqref{thm:equiv_condition 1}$\Rightarrow$\eqref{thm:equiv_condition 2}, we have the following corollary.
    \begin{corollary}\label{cor:q-trivial implies diagonal}
        A $\Q$-trivial generalized Bott manifold $B_h$ is weakly equivariantly diffeomorphic to $\prod_{i=1}^h \CP^{n_i}$ provided $n_i>1$ for all $i$.
    \end{corollary}
    \begin{proof}
        Since for each $i=1,\ldots,h$, $a_{1j}^i=\cdots=a_{nj}^i=0$ for all $j=1,\ldots,i-1$. Hence, the associated vector matrix of $B_h$ is block diagonal. Hence, the assertion is true.
    \end{proof}

    From Theorem~\ref{thm:equivalent conditions}, we have the following corollary.

    \begin{corollary}
        Let $M$ be a quasitoric manifold. If $H^\ast(M;\Q)$ is isomorphic to $H^\ast(\prod_{i=1}^h \CP^{n_i};\Q)$, then $M$ is homeomorphic to $\prod_{i=1}^h \CP^{n_i}$ provided $n_i>1$ for all $i$.
    \end{corollary}
    \begin{proof}
        By \cite{CMS10a}, if $H^\ast(M;\Q)$ is isomorphic to $H^\ast(\prod_{i=1}^h \CP^{n_i};\Q)$, then $M$ is homeomorphic to a generalized Bott manifold. But a $\Q$-trivial generalized Bott manifolds with $n_i>1$ is diffeomorphic to $\prod_{i=1}^h \CP^{n_i}$. Hence, $M$ is homeomorphic to $\prod_{i=1}^h \CP^{n_i}$.
    \end{proof}



    The following is the counter-example of Question~\ref{question:not Z-isomorphic}.
    \begin{example}\label{example:not Z-isomorphic}
        Let $B$ be a fiber bundle $P(\underline{\C}^3\oplus\xi)$ over $\CP^2$ and let $B'$ be a fiber bundle $P(\underline{\C}^3\oplus\xi^{\otimes 2})$ over $\CP^2$, where $\xi$ is the tautological line bundle over $\CP^2$. Let $y$ (respectively, $Y$) denote the negative of the first Chern class of the tautological line bundle over $B_2$ (respectively, $B'_2$). Then their cohomology rings are
        $$H^\ast(B)=\Z[x,y]/\langle x^3, y(y^3+xy^2)\rangle$$ and $$H^\ast(B')=\Z[X,Y]/\langle X^3, Y(Y^3+2XY^2)\rangle.$$
        Then the map $\phi$ defined by $\phi(x)=2X$ and $\phi(y)=Y$ is an isomorphism from $H^\ast(B;\Q)\to H^\ast(B';\Q)$.
        But this $\phi$ is not a $\Z$-isomorphism. Suppose that $\psi$ is an isomorphism $H^\ast(B;\Z)\to H^\ast(B';\Z)$.
        Then there exist $\alpha,\beta,\gamma,\delta$ in $\Z$ such that
        $$\left(\begin{array}{c}\psi(x)\\ \psi(y)\end{array}\right)
        =\left(\begin{array}{cc} \alpha & \beta\\ \gamma & \delta \end{array}\right)\left(\begin{array}{c}X\\Y\end{array}\right)$$
        and $\alpha\delta-\beta\gamma = \pm 1$. Since $\psi(x^3)=0$ in $H^\ast(B';\Z)$, we have $$(\alpha X+\beta Y)^3=\alpha^3X^3$$ as polynomials. So, we can see that $\beta$ is zero and $\alpha=\pm 1$, and hence $\delta=\pm 1$. Since $\psi(y(y^3+xy^2))$ is zero in $H^\ast(B';\Z)$, we have
        \begin{equation}\label{eqn:example1}
            (\gamma X+\delta Y)^3((\alpha+\gamma)X+\delta Y)=(aX+bY)X^3+cY(Y^3+2XY^2)
        \end{equation}
        as polynomials in $\Z[X,Y]$. By comparing the coefficients of $XY^3$ on both sides
        of~\eqref{eqn:example1}, we can see that
        \begin{equation}\label{eqn:example1-compare_coefficients}
            2c=3\gamma\delta^3+(\alpha+\gamma)\delta^3)=\delta(\alpha+4\gamma).
        \end{equation}
        Since the right hand side of~\eqref{eqn:example1-compare_coefficients} is odd, there is no such an integer $C$. Hence, there is no such $\Z$-isomorphism $\psi$.
    \end{example}
        Now consider $\Q$-trivial generalized Bott manifolds $B_h$ which have $\CP^1$-fibers, that is, $n_k=1$ for some $k\in [h]$.

    \begin{lemma}\label{lem:ordering of the tower}
        Let $B_h$ and $B_h'$ be two $h$-stage generalized Bott towers. If the associated vector matrices to them are
        \begin{equation*}
            A=\left(\begin{array}{ccccc}
            \mathbf{1}&&&&\\
            \ast&\ddots&&&\\
            \ast & \ast &\mathbf{1}&&\\
            \va_1&\cdots&\va_{h-2}&\mathbf{1}\\
            \vb_1&\cdots&\vb_{h-2}&\mathbf{0}&\mathbf{1}\end{array}
            \right)
            \mbox{ and }A'=\left(\begin{array}{ccccc}
            \mathbf{1}&&&&\\
            \ast&\ddots&&&\\
            \ast & \ast &\mathbf{1}&&\\
            \vb_1&\cdots&\vb_{h-2}&\mathbf{1}&\\
            \va_1&\cdots&\va_{h-2}&\mathbf{0}&\mathbf{1}\end{array}
            \right),
        \end{equation*} respectively, then $B_h$ and $B_h'$ are equivariantly diffeomorphic.
    \end{lemma}
    \begin{proof}
        Note that this lemma can be seen by the fact that $B_h$ and $B_h'$ are equivariantly diffeomorphic if two associated vector matrices are conjugated by a permutation matrix, see the paper \cite{CMS10a}. It is obvious that $$A'=E_\sigma AE_\sigma^{-1},$$ where $\sigma:=(1,\ldots,h-2,h,h-1)$ is the permutation on $[h]$ which permutes only $h-1$ and $h$.
    \end{proof}

    Now, we can prove Theorem~\ref{thm:Q-trivial gen. Bott}.
    \begin{proof}[Proof of Theorem~\ref{thm:Q-trivial gen. Bott}]
        Let $B_h$ be a $\Q$-trivial generalized Bott manifold whose associated matrix is of the form~\eqref{associated matrix}.

        Consider a map $\mu\colon\{1,\ldots,h\}\to\mathbb{N}$ given by $j\mapsto n_j$ and assume that the image of $\mu$ is the set $\{N_1,\ldots,N_m\}$ with $1=N_1<N_2<\cdots<N_m$.

        For each $i\in\mu^{-1}(1)$, by Proposition~\ref{prop:Q-trivial}, we have $c_1(\xi_i)^2=0$ in $H^\ast(B_h)$. Since $x_k^2\neq 0$ in $H^\ast(B_h)$ for $k\not\in\mu^{-1}(1)$, we can see that $a_{1k}^i=0$ for $k\in [i-1]$ with $n_k>1$.

        Now suppose that $n_j>1$. Then by Proposition~\ref{prop:Q-trivial}, we have the relation
        $$(n_j+1)^2c_2(\xi_j)=\frac{n_j(n_j+1)}{2}c_1(\xi_j)^2.$$ Since $x_k^2\neq 0$ in $H^\ast(B_h)$ for $n_k>1$, we can show that $\va_k^j=\mathbf{0}$ by using the same argument to the proof of Theorem~\ref{thm:equivalent conditions}.

        Since $\va_k^j=\mathbf{0}$ for all $n_k>1$, by Lemma~\ref{lem:ordering of the tower}, $B_h$ is diffeomorphic to the $\Q$-trivial generalized Bott manifold $B'$ whose associated matrix is of the form
        \begin{equation}\label{Q-tirivial associated matrix}
            (A')^T=\left(\begin{array}{cccc|cccc}
            1&&&&&&&\\
            a_{11}^2&1&&&&&&\\
            \vdots&\vdots&\ddots&&& &&\\
            a_{11}^r&a_{1,2}^r&\cdots&1&&&&\\
            \hline
            \va_1^{r+1}&\va_{2}^{r+1}&\cdots&\va_r^{r+1}&\mathbf{1}&&&\\
            \va_1^{r+2}&\va_{2}^{r+2}&\cdots&\va_r^{r+2}&\mathbf{0}&\mathbf{1}&&\\
            \vdots&\vdots&\cdots&\vdots&\vdots&\ddots&&\\
            \va_1^h&\va_{2}^h&\cdots&\va_r^h&\mathbf{0}&\cdots&\mathbf{0}&\mathbf{1}\\
             \end{array}\right),
        \end{equation} where $r$ is the cardinality of the set $\mu^{-1}(1)$, that is, $r=|\mu^{-1}(1)|$.
        This proves the theorem.
    \end{proof}

\end{document}